\documentclass[a4paper,11pt]{article}

\usepackage{anyfontsize}

\usepackage[margin=1.2in]{geometry}
\usepackage{amsmath,amsthm, amssymb, amsfonts, amsbsy}
\usepackage{thmtools} %
\usepackage{thm-restate} %
\usepackage{authblk}
\usepackage{booktabs}
\usepackage{colortbl}
\usepackage[pdftex,dvipsnames,table]{xcolor}
\usepackage{xfrac}
\usepackage{xargs}
\usepackage[utf8]{inputenc}
\usepackage[mathscr]{euscript}
\usepackage{xspace}
\usepackage{graphicx}
\usepackage{color}
\usepackage{enumerate}
\usepackage[english]{babel}
\usepackage{verbatim}
\usepackage{float}
\usepackage{soul}
\usepackage{rotating}
\usepackage{caption}
\usepackage[draft]{fixme}
\usepackage{longtable}
\usepackage{mathtools}
\usepackage[shortlabels]{enumitem}
\usepackage{tabularx}
\usepackage{bbm}
\usepackage[inkscapelatex=false]{svg}
\usepackage{subfig}
\usepackage{makecell}
\usepackage{xparse}

\definecolor{lightgray}{gray}{0.9}
\restylefloat{table}

\makeatletter
\newcommand{\multiline}[1]{%
  \begin{tabularx}{\dimexpr\linewidth-\ALG@thistlm}[t]{@{}X@{}}
    #1
  \end{tabularx}
}
\makeatother

\usepackage{algorithm}%
\usepackage[noend]{algpseudocode}%

\usepackage{hyperref}
\usepackage{cleveref}

\makeatletter
\newcounter{HALG@line}
\renewcommand{\theHALG@line}{\thealgorithm.\arabic{ALG@line}}
\makeatother

\hypersetup{
    colorlinks,
    linkcolor={red!50!black},
    citecolor={blue!50!black},
    urlcolor={blue!80!black}
}

\providecommand{\keywords}[1]{\textit{Keywords:} #1}
\usepackage{natbib}

\newcommand{\T}{\mathsf{\scriptscriptstyle T}} %

\newcommand{\R}{\mathbb{R}}
\newcommand{\Z}{\mathbb{Z}}

\newcommand{\B}{\mathbb{B}}
\newcommand{\Q}{\mathbb{Q}}

\newcommand{\mb}[1]{\mathbb{#1}}

\newcommand{\tsc}[1]{\textsc{#1}}

\newcommand{\mc}[1]{\mathcal{#1}}

\makeatletter
\newtheoremstyle{mystyle}%
{\topsep}{\topsep}
{\itshape}{}
{\bfseries}{}
{0.5em}
{\thmname{\@ifempty{#3}{#1}\@ifnotempty{#3}{#3}}}
\makeatother

\theoremstyle{mystyle}

\theoremstyle{plain}
\newtheorem{theorem}{Theorem}
\newtheorem{proposition}{Proposition}
\newtheorem{corollary}{Corollary}
\newtheorem{lemma}{Lemma}
\newtheorem{claim}{Claim}
\newtheorem{fact}{Fact}

\theoremstyle{definition}

\newtheorem{definition}{Definition}
\newtheorem{example}{Example}

\newcommand{\routes}{\hyperlink{def:routes}{\mc{R}}}
\newcommand{\vrpprob}[1]{\hyperlink{problem:vrp}{\tsc{vrp-prob}(#1)}}
\NewDocumentCommand{\fcvrp}{g}{
  \hyperlink{def:fcvrp}{
    f_{\tsc{cvrp}}\IfValueT{#1}{(#1)}
  }
}
\NewDocumentCommand{\f}{g}{
  \hyperref[def:rhs_function]{
    f\IfValueT{#1}{(#1)}
  }
}
\newcommand{\vrpform}[1]{\hyperref[form:cvrp]{\tsc{vrp-form}(#1)}}
\NewDocumentCommand{\fbrp}{g}{
  \hyperlink{def:fbrp}{
    f_{\tsc{brp}}\IfValueT{#1}{(#1)}
  }
}
\newcommand{\routesbrp}{\hyperlink{def:routes_brp}{\mc{R}_{\tsc{brp}}}}
\renewcommand{\P}[1]{\hyperref[def:polytope]{\mc{P}(#1)}}
\newcommand{\allforests}{\hyperlink{def:omega}{\Omega}}
\newcommand{\M}[1]{\hyperref[def:minimal_infeasible]{\mc{M}(#1)}}
\NewDocumentCommand{\lb}{m g}{
  \hyperref[def:lower_bound]{
    \ell_{#1}\IfValueT{#2}{(#2)}
  }
}
\renewcommand{\B}[1]{\hyperref[def:lower_bound]{\mc{B}(#1)}}
\NewDocumentCommand{\Phiset}{g}{
  \hyperref[def:phi_f]{
    \Phi\IfValueT{#1}{(#1)}
  }%
}
\NewDocumentCommand{\ub}{m g}{
  \hyperref[def:upper_bound]{
    u_{#1}\IfValueT{#2}{(#2)}
  }
}
\newcommand{\Qpath}{\hyperref[example:vrp]{\mc{Q}_{\tsc{path}}}}
\newcommand{\Cpath}{\hyperref[example:vrp]{\mc{C}_{\tsc{path}}}}
\newcommand{\Ft}[1]{\hyperref[example:vrp]{\mc{F}(#1)}}
\newcommand{\propstar}{\hyperref[item:path_star]{$(\star)$}}
\NewDocumentCommand{\Thetaset}{g}{
  \hyperref[def:theta]{
    \Theta\IfValueT{#1}{(#1)}
  }%
}
\NewDocumentCommand{\fg}{m g}{
  \hyperref[proposition:subadditive_rhs_function]{
    f_{#1}\IfValueT{#2}{(#2)}
  }%
}
\newcommand{\U}{\hyperlink{def:uncertainty_set}{\mc{U}}}

\newtheoremstyle{named}{}{}{}{}{\bfseries}{.}{.5em}{#1 #3}
\theoremstyle{named}

\usepackage{silence}
\begin{document}
\title{Representability of forests via generalized subtour elimination constraints}

\author[1]{Matheus J. Ota\thanks{(mjota@uwaterloo.ca)}}
\affil[1]{University of Waterloo, Waterloo, Ontario, Canada}

\maketitle

\begin{abstract}
Generalized subtour elimination constraints (GSECs) are widely used in state-of-the-art exact algorithms for vehicle routing and network design problems, as their right-hand sides often capture problem-specific feasibility conditions of each solution component. In this work, we present the first characterization of the families of forests that can be represented as the integer points inside a polytope defined by GSECs. This result generalizes a recent framework developed for vehicle routing problems under uncertainty and broadens the applicability of GSEC-based formulations to a wider class of combinatorial problems. In particular, using our characterization, we recover vehicle routing formulations that could not be obtained with previous results. Additionally, we show that GSECs can naturally model a robust variant of the capacitated minimum spanning tree problem.
\end{abstract}

\noindent
\keywords{representability, generalized subtour elimination constraints, network design, vehicle routing, branch-and-cut}

\section{Introduction}
\label{section:intro}

Several exact algorithms for network design and vehicle routing problems (VRPs) use edge-based integer programming (IP) formulations with \emph{generalized subtour elimination constraints} (GSECs)~\cite{TothV02, hall1996, dell2014bike, gounaris2013robust, Dinh2018, ghosal2020, ghosal2024, laporte1986generalized, laporte1983branch}. In these formulations, the right-hand side (RHS) values of the GSECs encode the problem-specific feasibility conditions of each component of a solution. Several of these problems are then solved using roughly the same branch-and-cut algorithm, with the only modification being in the RHS coefficients of the separated GSECs~\cite{TothV02, ghosal2024, Dinh2018, gounaris2013robust, dell2014bike}.

In the context of vehicle routing problems with stochastic demands (VRPSDs), Ghosal et al.~\cite{ghosal2024} recently proposed a framework that provides \emph{sufficient conditions} under which a set of feasible routes can be modeled using GSECs. More precisely, they have shown that a branch-and-cut approach based on GSECs can be applied whenever the set of feasible routes is both \emph{downward closed} and \emph{permutation invariant}. Using these conditions, they presented a unified argument for several existing exact algorithms for VRPSDs, including the robust VRPSD~\cite{gounaris2013robust}, the chance-constrained VRPSD~\cite{Dinh2018}, and the distributionally robust VRPSD~\cite{ghosal2020}. However, their framework does not provide a full characterization, as they also prove that these conditions are \emph{not necessary}. Indeed, as we show later in Section~\ref{section:setup}, the set of feasible routes for the bike sharing rebalancing problem (BRP) is not permutation invariant, but Dell'Amico et al.~\cite{dell2014bike} still successfully model this problem using GSECs.

To better understand the modeling power of GSECs, we abstract away the degree constraints in VRP formulations and we focus on identifying the families of forests that can be represented solely with GSECs. As a result of this study, we generalize and extend the framework of Ghosal et al.~\cite{ghosal2024} in several ways. Specifically, we make the following contributions.
\begin{itemize}
    \item We establish the first characterization (sufficient and necessary conditions) for determining whether a given family of forests can be represented using only GSECs (Section~\ref{subsection:blocking_property}). Additionally, when the conditions of this characterization are satisfied, we precisely identify the set of RHS values that can be used in the GSECs to represent the given family of forests (Section~\ref{subsection:rhs}).
    \item Based on this result, we also characterize the families of forests that can be represented using GSEC together with some additional constraints (Section~\ref{section:extension}). This result generalizes the conditions of Ghosal et al.~\cite{ghosal2024}, since it implies that a certain \emph{minimal infeasibility property} is a weaker condition than permutation invariance, yet it is still sufficient to model VRPs using GSECs.
    \item We provide a simplified approach to verify whether certain families of forests satisfy our conditions, and applying this method, we obtain GSEC-based formulations that could not be recovered with the framework of Ghosal et al.~\cite{ghosal2024} (Section~\ref{section:applications}). In particular, we recover the BRP formulation of Dell'Amico et al.~\cite{dell2014bike}, and we derive a GSEC-based formulation for a robust variant of the capacitated minimum spanning tree (CMST) problem~\cite{hall1992polyhedral, hall1996, uchoa2008robust}.
\end{itemize}

\noindent
Overall, these results highlight the potential of GSECs as a versatile modeling tool for a broader class of combinatorial optimization problems, extending the framework of Ghosal et al.~\cite{ghosal2024} beyond VRPSDs and even VRPs. \\

\noindent
\textbf{Notation:} Let~$G = (V, E)$ be an undirected graph. For each~$v \in V$,~$\delta(v)$ denotes the set of edges incident to~$v$ (or~$\delta_G(v)$ when the graph must be specified). For each~$S \subseteq V$,~$E(S)$ refers to the set of edges with both endpoints in~$S$. The notation~$H \subseteq G$ indicates that~$H$ is a subgraph of~$G$. If~$C_1, \ldots, C_t \subseteq G$ are the (connected) components of~$H$, then we may write~$H = \{C_1, \ldots, C_t\}$. In particular, if~$H$ is a forest,~$t = |H| = |V(H)| - |E(H)|$ (if~$H$ is the empty graph, then~$|H| = 0$). We represent any path~$P \subseteq G$ by a tuple of vertices~$(v_1, \ldots, v_t)$, that is, if~$P = (v_1, \ldots, v_t) \subseteq G$, then~$E(P) = \{v_i v_{i+1} : i \in [t-1]\}$ (we assume~$[0] = \emptyset$). Paths are always assumed to be simple, and~$P = (v_1, \ldots, v_t) = (v_t, \ldots, v_1)$.

For any vector (respectively, function)~$g$, we use~$g_i$ and~$g(i)$ interchangeably, and for any subset~$U$ of its indices (respectively, domain), we define~$g(U) \coloneqq \sum_{i \in U} g(i)$. For any vector~$x \in \R^E$ and~$S \subseteq V$, we use the notation~$x(S) \coloneqq x(E(S)) = \sum_{e \in E(S)} x_e$. Moreover, for each~$E' \subseteq E$, ~$x|_{E'}$ refers to the restriction (or projection) of~$x$ onto~$\R^{E'}$. Finally, for any~$H \subseteq G$, the vector~$\mathbbm{1}_H \in \R^E$ denotes the characteristic vector of~$E(H)$.

\section{Generalized subtour elimination constraints and edge-based formulations for vehicle routing problems}
\label{section:setup}

We begin our discussion with the definition of a generic class of VRPs. We are given a complete undirected graph~$G_0 = (V_0, E_0)$ with edge costs~$c \in \Q^{E_0}_+$. The set of vertices~$V_0 = \{0\} \dot\cup V$ contains a special vertex denoted~$0 \in V_0$ representing the \emph{depot}, while the remaining vertices in~$V$ represent the \emph{customers}. We have~$k \in \Z_{++}$ available vehicles, all located initially at the depot. For convenience, we also set~$G = G_0 - \{0\} = (V, E)$, that is,~$G$ is the graph obtained from~$G_0$ by removing the depot.

Let \hypertarget{def:routes}{$\mc{R}$} be a family of feasible paths (or routes) in~$G$. We define problem \hypertarget{problem:vrp}{$\tsc{vrp-prob}(\mc{R})$} as follows. Let~$\mc{S} = \{C_1, \ldots, C_k\}$ be a set of~$k$ (simple) cycles, where each~$C_i \subseteq G_0$ contains the depot. Let~$P_i = C_i - \{0\} \subseteq G$ be the path obtained from~$C_i$ by deleting the depot. We say that~$\mc{S}$ is \emph{feasible} for~$\vrpprob{\mc{R}}$ if each~$P_i$ belongs to~$\routes$ and~$\{V(P_i)\}_{i \in [k]}$ forms a partition of~$V$. The objective of~$\vrpprob{\mc{R}}$ is to find a feasible solution~$\mc{S} = \{C_1, \ldots, C_k\}$ that minimizes the edge costs~$\sum_{i \in [k]} c(E_0(C_i))$.

All the VRPs mentioned in the introduction~\cite{laporte1983branch, gounaris2013robust, Dinh2018, ghosal2020, ghosal2024, dell2014bike} can be expressed as problems in the form of~$\vrpprob{\mc{R}}$.%
\footnote{
Some of these VRPs may allow solutions that use at most~$k$ cycles, and our reasoning extends naturally to this case. On the other hand, we do not address here VRPs that allow solutions visiting customers more than once, such as the VRPs with \emph{splittable demands}~\cite{archetti2008split}.}
For instance, the classical capacitated vehicle routing problem (CVRP) corresponds to problem~$\vrpprob{\mc{R}_{\tsc{cvrp}}}$, where~$\mc{R}_{\tsc{cvrp}} = \{\text{path~$P \subseteq G$} : d(V(P)) \leq C\}$,~$d \in \Q^V_+$ is a vector of customer demands, and~$Q \in \Q_+$ is the vehicle capacity. 

For any subset~$\emptyset \subsetneq S \subseteq V$, define \hypertarget{def:fcvrp}{$f_{\textsc{cvrp}}(S) \coloneqq \max\{1, \lceil d(S)/Q \rceil\}$}. A well-known CVRP formulation due to Laporte and Nobert~\cite{laporte1983branch} is as follows:
\begin{subequations}
\label{form:cvrp}
\begin{align}
    \tsc{vrp-form}(\fcvrp) \quad \quad \quad \min~~ & c^\T x \\
    \text{s.t.~~} & x(\delta_{G_0}(0)) = 2k, \label{eq:cvrp_depot} \\
    & x(\delta_{G_0}(v)) = 2, & \forall v \in V \label{eq:cvrp_customers} \\
    & x(S) \leq |S| - \fcvrp{S}, & \forall \emptyset \subsetneq S \subseteq V, \label{ineq:cvrp_gsec} \\
    & x_e \leq 1 + \mb{I}(0 \in e), & \forall e \in E_0, \\
    & x \in \Z^{E_0}_+,
\end{align}
\end{subequations}
where~$\mb{I}(\,\cdot\,)$ denotes the indicator function. Inequalities~\eqref{ineq:cvrp_gsec} are named GSECs, since they reduce to the standard \emph{subtour elimination constraints} (SECs) whenever~$\fcvrp{S} = 1$, for all~$\emptyset \subsetneq S \subseteq V$ (see~\cite{TothV02, laporte1986generalized}).

To associate Formulation~\eqref{form:cvrp} with different VRPs, we introduce the following definition.
\begin{definition}
    \label{def:rhs_function}
    A function~$f : 2^{V} \to \Z_+$ is a \emph{RHS function} if~$f(S) \in \{1, \ldots, |S|\}$, for every~$\emptyset \subsetneq S \subseteq V$. For convenience, the function~$f$ also satisfies~$f(\emptyset) = 0$.
\end{definition}
IP formulations for many VRPSD variants~\cite{gounaris2013robust, Dinh2018, ghosal2020, ghosal2024} can be obtained from Formulation~\eqref{form:cvrp} by just replacing the RHS function~$\fcvrp$.%
\footnote{Certain formulations instead use inequalities~$x(\delta_{G_0}(S)) \geq 2 \f{S}$. These inequalities can be shown to be equivalent to the GSECs by summing the degree constraints~\eqref{eq:cvrp_customers} over all~$v \in S$. When the formulation is instead defined on a complete directed graph~$D = (\{0\}\cup V, A)$ (with~$x \in \R^A$), the analogous inequalities~$x(\delta_D^+(S)) \geq \f{S}$ are also equivalent to the GSECs~$x(S) \le |S| - \f{S}$.}
The framework of Ghosal et al.~\cite{ghosal2024} partially explains this phenomenon, as they establish that~$\vrpprob{\routes}$ can be solved with~$\vrpform{f}$ (for some RHS function~$\f$) whenever the family of feasible paths~$\routes$ contains every path with at most one vertex and is both \emph{downward closed} and \emph{permutation invariant}, which we formally define next. 

For convenience in the following sections, we present these properties with respect to a general family of subgraphs~$\mc{H}$, rather than the family of feasible paths~$\routes$. Since general graphs cannot be written as tuples (as in the case of paths), we replace the term permutation invariance with \emph{vertex-consistency}.

\begin{definition}
    \label{def:downward_closed}
    A family of subgraphs~$\mc{H}$ of~$G$ is \emph{downward closed} if, for every~$F \in \mc{H}$ and~$F' \subseteq F$, we have that~$F' \in \mc{H}$.
\end{definition}

\begin{definition}
    \label{def:vertex_consistency}
    A family of subgraphs~$\mc{H}$ of~$G$ is \emph{vertex-consistent} if, for every~$F, F' \subseteq G$ with~$V(F) = V(F')$, we have that~$F \in \mc{H}$ if and only if~$F' \in \mc{H}$.
\end{definition}

\noindent
In fact, when applied to the family of feasible paths~$\routes$, the downward closedness property of Ghosal et al.~\cite{ghosal2024} is stronger than Definition~\ref{def:downward_closed}, as it states that if~$P = (v_1, \ldots, v_\ell) \in \routes$, then~$P' = (v_{i_1}, \ldots, v_{i_t}) \in \routes$, for every~$1 \leq i_1 < \ldots < i_t \leq \ell$. One can verify, however, that if~$\routes$ satisfies Definition~\ref{def:vertex_consistency}, then the two properties are equivalent.

As pointed out in Section~\ref{section:intro}, although the framework of Ghosal et al.~\cite{ghosal2024} unifies several VRPSD variants, it does not capture the GSEC-based formulation of Dell’Amico et al.~\cite{dell2014bike} for the BRP (or the 1-commodity pickup-and-delivery TSP formulation of~\cite{hernandez2003one}). In this problem, denoted~$\vrpprob{\mc{R}_{\tsc{brp}}}$, the vehicle load represents bikes, while the demands correspond to the number of bikes that must be picked up or delivered at each station. Vehicles are located at a central depot and start their routes with an initial load between~$0$ and~$Q \in \Q_+$. Customer demands~$d_v$ can be positive or negative (with~$|d_v| \leq Q$), and the accumulated load along a route must always remain within the interval~$[0, Q]$. Dell'Amico et al.~\cite{dell2014bike} show that the BRP can be formulated as~$\vrpform{f_{\tsc{brp}}}$, where, for every~$\emptyset \subsetneq S \subseteq V$, \hypertarget{def:fbrp}{$f_{\tsc{brp}}(S) \coloneqq \max\{1, \lceil |d(S)| / Q \rceil\}$}.

Formally, a path~$P = (v_1, \ldots, v_\ell)$ belongs to \hypertarget{def:routes_brp}{$\mc{R}_{\tsc{brp}}$} if and only if there exists an initial load~$q \in [0, Q] \cap \Q$ such that~$0 \leq q + \sum_{j \in [i]} d_{v_i} \leq Q$, for all~$i \in [\ell]$. Or equivalently (see Proposition~\ref{prop:brp}), if there exists~$q'$ such that~$0 \leq q' + \sum_{j = i}^\ell d_{v_i} \leq Q$, for all~$i \in [\ell]$. The following simple example shows that~$\routesbrp$ may not be vertex-consistent.

\begin{example}
\label{example:brp}
Suppose that~$Q = 1$,~$k = 1$ and we only have three customers~$v_1$,~$v_2$ and~$v_3$, with demands~$d_{v_1} = 1$,~$d_{v_2} = 1$ and~$d_{v_3} = -1$. Consider a route that starts at the depot and visits customers~$v_1, v_2, v_3$, in this order. This route corresponds to path~$P = (v_1, v_2, v_3)$, which does not belong to~$\routesbrp$, since~$d_{v_1} + d_{v_2} = 2 > 1 = Q$. On the other hand,~$P' = (v_1, v_3, v_2) \in \routesbrp$, as the vehicle can leave the depot with zero initial load and the accumulated load stays within~$[0, Q]$ at all times.~\qed
\end{example}

Example~\ref{example:brp} raises the question of which problems of the form~$\vrpprob{\mc{R}}$ can be modeled using GSECs but are not captured by the sufficient conditions of Ghosal et al.~\cite{ghosal2024}. To investigate this further, in Section~\ref{section:representation}, we drop the degree constraints~\eqref{eq:cvrp_depot} and~\eqref{eq:cvrp_customers} from Formulation~\eqref{form:cvrp}, and we characterize the forests that can be represented solely with the GSECs and the edge upper-bound constraints. Once these tools are developed, we reintroduce the degree constraints in Section~\ref{section:extension}.

\section{Representable families of forests}
\label{section:representation}

From now on, we fix~$G = (V, E)$ to be an arbitrary undirected graph (which may not be the same as the graph~$G_0 - \{0\}$ from Section~\ref{section:setup}). For any RHS function~$\f$, we define the polytope
\begin{equation}
    \label{def:polytope}
    \tag{$\mc{P}(f ; G)$}
    \mc{P}(f ; G) \coloneqq \{x \in [0, 1]^{E} : x(S) \leq |S| - \f{S},~~\forall \emptyset \subsetneq S \subseteq V\}.
\end{equation}
Since the graph~$G$ is fixed, we sometimes omit the dependence on~$G$ from the notation. In particular, we write~$\P{f}$ instead of~$\P{f; G}$. Additionally, to avoid repeating ourselves, whenever we write~$\P{f}$, it is implicitly assumed that~$\f$ is a RHS function.

The GSECs imply the SECs, so the integer vectors inside~$\P{f}$ correspond to forests in~$G$ (for any RHS function~$\f$). We use \hypertarget{def:omega}{$\allforests$} to denote the family of all forests in~$G$. The notion of \emph{representability} is formalized as follows.
\begin{definition}
    \label{def:representation}
    Let~$\mc{F}$ be a family of forests in~$G$ and let~$\mc{P} \subseteq \R^E$.
    We say that~$\mc{P}$ \emph{represents}~$\mc{F}$ if~$\mc{P} \cap \Z^{E} = \{\mathbbm{1}_{F} : F \in \mc{F}\}$. Furthermore,~$\mc{F}$ is \emph{representable} if there exists a RHS function~$\f$ such that~$\P{f}$ represents~$\mc{F}$.
\end{definition}

Definition~\ref{def:representation} identifies each forest in~$\mc{F}$ with its edge set. However, different forests may share the same edge set, as they might differ only by a set of isolated vertices. Consequently, the same set~$\P{f}$ may represent two distinct families of forests~$\mc{F}$ and~$\mc{F}'$, as long as their incidence vectors coincide. In this way, we often focus on \emph{edge-consistent} families of forests:
\begin{definition}
    \label{def:edge-consistency}
    A family of forests~$\mc{F}$ is \emph{edge-consistent} if, for every pair of forests~$F$ and~$F'$ in~$G$ with~$E(F) = E(F')$, we have that~$F \in \mc{F}$ if and only if~$F' \in \mc{F}$.
\end{definition}

The assumption of edge-consistency is without loss of generality: given any family of forests~$\mc{F}'$, one can always construct a unique edge-consistent family~$\mc{F}$ such that~$\{\mathbbm{1}_{F} : F \in \mc{F}\} = \{\mathbbm{1}_{F} : F \in \mc{F}'\}$. Furthermore, Definition~\ref{def:edge-consistency} is convenient for how we express our characterization, since we associate each forest~$F \in \mc{F}$ with both its subgraphs (see Fact~\ref{fact:downward}) and its vertex set~$V(F)$ (see Definition~\ref{def:lower_bound}). Definition~\ref{def:edge-consistency} thus ensures that forests with identical edge sets but different vertex sets are treated in the same way.

\subsection{The main characterization}
\label{subsection:blocking_property}

We start by deriving necessary conditions for an edge-consistent family of forests~$\mc{F}$ to be representable. Thus, let us assume for the moment that~$\P{f}$ represents~$\mc{F}$.

By the definition of RHS functions (Definition~\ref{def:rhs_function}),~$\P{f}$ contains the origin, so we immediately obtain the following fact.
\begin{fact}
    \label{fact:trivial_forest}
    If~$\mc{F}$ is an edge-consistent representable family of forests, then~$\mc{F}$ contains all the forests with no edges (including the empty graph).
\end{fact}

\noindent
Another simple observation is that if~$x \in \P{f}$, then every~$y \leq x$ (componentwise) also belongs to~$\P{f}$, so~$\mc{F}$ must be downward closed (Definition~\ref{def:downward_closed}):
\begin{fact}
    \label{fact:downward}
    If~$\mc{F}$ is an edge-consistent representable family of forests, then~$\mc{F}$ is downward closed.
\end{fact}

\noindent
Note that if~$\mc{F}$ is representable but not edge-consistent, we may have that~$F \in \mc{F}$ while~$F' \subsetneq F$ does not belong to~$\mc{F}$ (but there exists~$F'' \in \mc{F}$ such that~$\mathbbm{1}_{F''} = \mathbbm{1}_{F'} \leq \mathbbm{1}_F$).

Fact~\ref{fact:downward} implies that feasible forests cannot contain \emph{minimal infeasible forests}, defined as follows. 

\begin{definition}
    \label{def:minimal_infeasible}
    Let~$\mc{F}$ be a family of forests. A forest~$F \subseteq G$ is \emph{minimal infeasible} with respect to~$\mc{F}$ if~$F \notin \mc{F}$ and every proper subgraph~$F' \subsetneq F$ belongs to~$\mc{F}$. The notation~$\mc{M}(\mc{F})$ denotes the set of all such minimal infeasible forests.
\end{definition}

Definition~\ref{def:minimal_infeasible} is key for our characterization of (edge-consistent) representable families of forests. For an intuition of why this is the case, consider Example~\ref{example:brp}: path~$(v_1, v_2, v_3)$ is infeasible but not minimal; on the other hand, path~$(v_1, v_2)$ is minimal infeasible, and no feasible solution can cover customers~$\{v_1, v_2\}$ using a single path (or route). This simple example illustrates that the forests in~$\M{\mc{F}}$ may impose lower bounds on the number of components that some feasible forests can have. Motivated by this observation, we introduce the following definitions. 

We remark that Definition~\ref{def:lower_bound} is stated with respect to a generic family of forests~$\mc{C} \subseteq \allforests$, as this will be convenient in later sections. However, for now, we always assume that~$\mc{C} = \allforests$, and in this case, we omit the~$\allforests$ for simplicity (so we write~$\ell_{\mc{F}}$ and~$\mc{B}(\mc{F})$ instead of~$\ell_{\mc{F}, \allforests}$ and~$\mc{B}(\mc{F}, \allforests)$).
\begin{definition}
    \label{def:lower_bound}
    Let~$\mc{F}$ and~$\mc{C}$ be two families of forests. Define~$\mc{B}(\mc{F}, \mc{C}) \coloneqq \{V(F) : F \in \M{\mc{F}} \cap \mc{C}\}$ and
    \begin{equation}
        \tag{$\ell_{\mc{F}, \mc{C}}$}
        \ell_{\mc{F}, \mc{C}}(S) \coloneqq 1 +
        \begin{cases}
            \max\{|F| : F \in \M{\mc{F}} \cap \mc{C}, V(F) = S\}, & \text{if~$S \in \mc{B}(\mc{F}, \mc{C})$,}\\
            0, & \text{otherwise.}
        \end{cases}
    \end{equation}
    for each~$\emptyset \subsetneq S \subseteq V$. For convenience,~$\ell_{\mc{F}, \mc{C}}(\emptyset) = 0$.
\end{definition}

\begin{definition}
\label{def:blocking_property}
    A family of forests~$\mc{F}$ has the \emph{minimal infeasibility property} if every~$F \in \mc{F}$ satisfies~$|F| \geq \lb{\mc{F}}{V(F)}$.
\end{definition}

\begin{lemma}
    \label{lemma:minimal_necessary}
    If~$\mc{F}$ is an edge-consistent representable family of forests, then~$\mc{F}$ has the minimal infeasibility property.
\end{lemma}
\begin{proof}
    Let~$\P{f}$ represent~$\mc{F}$ and suppose by contradiction that~$F \in \mc{F}$ is such that~$|F| \leq \lb{\mc{F}}{V(F)} - 1$. Since~$\lb{\mc{F}}{\emptyset} = 0$, we know that~$|F| \geq 1$, meaning that~$\lb{\mc{F}}{V(F)} \geq 2$. Definition~\ref{def:lower_bound} implies that there exists a forest~$H \in \M{\mc{F}}$ with~$V(H) = V(F)$ such that~$|H| = \lb{\mc{F}}{V(F)} - 1$. 
    
    By edge-consistency of~$\mc{F}$, no forest~$F' \in \mc{F}$ has the same edge set as~$H \notin \mc{F}$, so~$\mathbbm{1}_{H} \notin \{\mathbbm{1}_{F'} : F' \in \mc{F}\}$. We thus show the desired contradiction by proving that~$\mathbbm{1}_{H} \in \P{f} \cap \Z^E$. To do this, we use case analysis to verify that~$\mathbbm{1}_{H}$ satisfies the GSEC~$x(S) \leq |S| - \f{S}$, for every~$\emptyset \subsetneq S \subseteq V$.\\
    \noindent
    \textbf{Case~$|V(F) \cap S| = 0$:} Since~$V(H) = V(F)$, it follows from the definition of RHS functions  that~$0 = \mathbbm{1}_{H}(S) \leq |S| - \f{S}$. 

    \noindent
    \textbf{Case~$|V(F) \cap S| = |V(F)|$:} Since~$\P{f}$ represents~$\mc{F}$ and~$F \in \mc{F}$, we know that~$\mathbbm{1}_F \in \P{f}$. Hence,~$$\mathbbm{1}_{H}(S) = |V(H)| - (\lb{\mc{F}}{V(F)} - 1) \leq |V(F)| - |F| = \mathbbm{1}_{F}(S) \leq |S| - \f{S}.$$ 
    
    \noindent
    \textbf{Case~$0 < |V(F) \cap S| < |V(F)|$:} Let~$H'$ be the forest obtained by deleting from~$H$ all the vertices that are not in~$S$. By minimality of~$H$,~$H'$ belongs to~$\mc{F}$ and~$\mathbbm{1}_{H'} \in \P{f}$ (by representability). Hence,~$\mathbbm{1}_{H}(S) = \mathbbm{1}_{H'}(S) \leq |S| - \f{S}$.
\end{proof}

Although not used in our development, it is worth noting that, if~$\mc{F}$ satisfies the minimal infeasibility property, then all minimal infeasible forests spanning a given set of vertices have the same number of components.
\begin{claim}
    \label{claim:minimal_same_size}
    Let~$\mc{F}$ be a family of forests satisfying the minimal infeasibility property. Then, for every~$F, F' \in \M{\mc{F}}$ with~$V(F) = V(F')$, we have that~$|F| = |F'|$. 
\end{claim}
\begin{proof}
    Suppose by contradiction that~$F, F' \in \M{\mc{F}}$ are such that~$V(F) = V(F')$ and~$|F| < |F'|$. Since~$F$ and~$F'$ are minimal infeasible,~$F \subsetneq F'$ and~$F' \subsetneq F$, which implies that there exists~$e \in E(F) \setminus E(F')$. Let~$H$ be obtained from~$F$ by deleting edge~$e$, i.e.,~$V(H) = V(F)$ and~$E(H) = E(F) \setminus \{e\}$. By minimal infeasibility of~$F$,~$H$ is feasible. Moreover,~$|H| = |F| + 1 \leq |F'| \leq \lb{\mc{F}}{V(F)} - 1$, contradicting the minimal infeasibility property.
\end{proof}

Next, for any set function~$f : 2^V \to \R$, we define
\begin{equation}
    \label{def:phi_f}
    \tag{$\Phi(f)$}
    \Phi(f) \coloneqq \{F \in \allforests : |F'| \geq f(V(F')),~\forall F' \subseteq F \}. 
\end{equation}

\noindent
Using this notation, Definitions~\ref{def:downward_closed} and~\ref{def:blocking_property} can be concisely expressed as follows.
\begin{lemma}
    \label{lemma:phi_lower_bound}
    Let~$\mc{F}$ be a family of forests. Then~$\mc{F}$ is nonempty, downward closed and has the minimal infeasibility property if and only if~$\mc{F} = \Phiset{\lb{\mc{F}}}$.
\end{lemma}
\begin{proof}
    By the definition of~$\Phiset$, it is clear that if~$\mc{F} = \Phiset{\lb{\mc{F}}}$, then~$\mc{F}$ is nonempty, downward closed and has the minimal infeasibility property (note that~$\Phiset{\lb{\mc{F}}}$ always contains the empty graph). To show the other direction, let~$F$ be an arbitrary forest in~$G$.

    Since~$\mc{F}$ is nonempty and downward closed, it follows that~$\emptyset \in \mc{F}$. Hence, whenever~$F \notin \mc{F}$ there exists~$F' \subseteq F$ such that~$F' \in \M{\mc{F}}$. By the definition of~$\lb{\mc{F}}$,~$|F'| \leq \lb{\mc{F}}{V(F')} - 1$, meaning that both~$F'$ and~$F$ do not belong to~$\Phiset{\lb{\mc{F}}}$. To show the other direction of the inclusion, suppose that~$F \in \mc{F}$. Downward closedness implies that every~$F' \subseteq F$ belongs to~$\mc{F}$, while the minimal infeasibility property gives~$|F'| \geq \lb{\mc{F}}{V(F')}$. This proves that~$F \in \Phiset{\lb{\mc{F}}}$.
\end{proof}

Moreover,~$\Phiset{f}$ is representable whenever it is edge-consistent (and~$f$ is a RHS function).
\begin{lemma}
    \label{lemma:phi_representable}
    Let~$\f$ be a RHS function and suppose that~$\Phiset{f}$ is edge-consistent. Then~$\P{f}$ represents~$\Phiset{f}$.
\end{lemma}
\begin{proof}
    By the definition of representability (Definition~\ref{def:representation}), we need to prove that~$\P{f} \cap \Z^E = \{\mathbbm{1}_F : F \in \Phiset{f}\}$. Let~$F$ be an arbitrary forest in~$G$ and suppose first that~$F \notin \Phiset{f}$, so there exists~$F' \subseteq F$ such that~$|F'| < \f{V(F')}$. Then
    $$\mathbbm{1}_F(V(F')) \geq \mathbbm{1}_{F'}(V(F')) = |V(F')| - |F'| > |V(F')| - \f{V(F')},$$
    meaning that~$\mathbbm{1}_F \notin \P{f}$. 
    
    For the converse, suppose that~$F \in \Phiset{f}$. Our goal is to show that~$\mathbbm{1}_F \in \P{f}$. Take an arbitrary set~$\emptyset \subsetneq S \subseteq V$. By edge-consistency, we can add singletons to~$F$ to obtain~$H \supseteq F$ such that~$H \in \Phiset{f}$,~$S \subseteq V(H)$ and~$\mathbbm{1}_F = \mathbbm{1}_H$. Let~$F'$ be obtained from~$H$ by deleting the vertices that are not in~$S$. Since~$\Phiset{f}$ is downward closed, we know that~$|F'| \geq \f{V(F')}$, therefore,~$$\mathbbm{1}_F(S) = \mathbbm{1}_H(S) = \mathbbm{1}_{F'}(S) = |S| - |F'| \leq |S| - \f{S},$$ as desired.
\end{proof}

Combining Lemmas~\ref{lemma:phi_lower_bound} and~\ref{lemma:phi_representable} we obtain the following characterization of an edge-consistent representable family of forests.

\begin{theorem}
    \label{thm:characterization1}
    Let~$\mc{F}$ be an edge-consistent family of forests. Then~$\mc{F}$ is representable if and only if~$\mc{F} = \Phiset{\lb{\mc{F}}}$.
\end{theorem}
\begin{proof}
    Since~$\mc{F}$ is representable, it follows from Facts~\ref{fact:trivial_forest} and~\ref{fact:downward} and Lemma~\ref{lemma:minimal_necessary}, that~$\mc{F}$ is nonempty, downward closed and has the minimal infeasibility property. Lemma~\ref{lemma:phi_lower_bound} then yields~$\mc{F} = \Phiset{\lb{\mc{F}}}$. Conversely, Lemma~\ref{lemma:phi_representable} implies that~$\P{\lb{\mc{F}}}$ represents~$\mc{F}$.
\end{proof}

Lastly, before continuing our discussion, we offer two simple examples illustrating how Theorem~\ref{thm:characterization1} applies to families of forests that can and cannot be represented with GSECs.

\begin{example}
\label{example:cmst}
Let~$Q \in \Q_+$ and $d \in \Q^V_{+}$ be such that~$d_v \leq Q$ for all~$v \in V$. Consider the family of forests~$\mc{F}_{\tsc{cmst}} = \{ F \in \allforests : d(V(T)) \leq Q,~\forall \, \text{tree } T \in F \}$.
Clearly,~$\mc{F}_{\tsc{cmst}}$ is downward closed and contains all the forests with no edges. Moreover, any minimal infeasible forest with respect to~$\mc{F}_{\tsc{cmst}}$ is a tree~$T \subseteq G$ such that~$d(V(T)) > Q$. Therefore, for every~$\emptyset \subsetneq S \subseteq V$,~$\lb{\mc{F}_{\tsc{cmst}}}{S} \leq 2$, and equality implies that~$d(S) > Q$, so the vertices in~$S$ cannot be covered with a single tree. We thus conclude that~$\mc{F}_{\tsc{cmst}}$ satisfies the minimal infeasibility property, and by Theorem~\ref{thm:characterization1},~$\mc{F}_{\tsc{cmst}}$ is representable. This example is consistent with previous work showing that GSECs can be used to formulate the CMST~\cite{hall1996}.~\qed
\end{example}

\begin{example}
\label{example:degree}
Let~$b \in \Z^V_+$ be a vector of upper bounds on the degree of each vertex, and consider the family of forests~$\mc{F}_{\tsc{deg}} = \{ F \in \allforests : |\delta_F(v)| \leq b_v,~\forall v \in V(F) \}$. The family~$\mc{F}_{\tsc{deg}}$ is downward closed and contains all the forests with no edges. However, perhaps not surprisingly,~$\mc{F}_{\tsc{deg}}$ cannot be represented with GSECs. To see this, suppose that~$G$ is the complete graph,~$V = \{v_1, v_2, v_3, v_4\}$, and~$b_v = 2$, for all~$v \in V$. Consider the spanning trees~$T_1$ and~$T_2$ with~$E(T_1) = \{v_1v_2, v_1v_3, v_1v_4\}$ and~$E(T_2) = \{v_1v_2, v_2v_3, v_3v_4\}$. Since~$T_1$ is minimally infeasible, we have~$\lb{\mc{F}_{\tsc{deg}}}{V} \ge 2$. However,~$\mathbbm{1}_{T_2}(V) = 3 > |V| - 2 = 2$, which shows that the minimal infeasibility property fails.~\qed
\end{example}

\subsection{Different RHS functions}
\label{subsection:rhs}

Although Theorem~\ref{thm:characterization1} precisely identifies the conditions on a family of forests that guarantee its representability via GSECs, the set~$\P{\lb{\mc{F}}}$ may provide a weak polyhedral relaxation of the convex hull of~$\{\mathbbm{1}_F : F \in \mc{F}\}$. This weakness can be particularly undesirable when using the relaxation~$\P{\lb{\mc{F}}}$ in a branch-and-cut algorithm. In this sense, we now assume that~$\mc{F}$ is representable, and we ask which choices of RHS functions~$\f$ ensure that~$\P{f}$ represents~$\mc{F}$. 

Our first result shows that~$\P{\lb{\mc{F}}}$ is the weakest relaxation of this type. We remark that the following statements are presented in a somewhat general form, as this will be useful to prove the results in Section~\ref{section:extension}.
\begin{lemma}
    \label{lemma:weakest_relaxation}
    Let~$\mc{F}$ be a family of forests and let~$f$ be a RHS function such that~$\{\mathbbm{1}_F : F \in \mc{F}\} \subseteq \P{f}$. Then
    \begin{enumerate}[(a), leftmargin=*, align=left]
        \item for any forest~$F \in \M{\mc{F}}$,~$\mathbbm{1}_F \notin \P{f}$ implies that~$x(V(F)) \leq |V(F)| - |F| - 1$ is valid for~$\P{f}$; and \label{item:weakest_relaxation1}
        \item if~$\P{f}$ represents~$\mc{F}$, then~$\P{f} \subseteq \P{\lb{\mc{F}}}$. \label{item:weakest_relaxation2}
    \end{enumerate}
\end{lemma}
\begin{proof}
    To show item~\ref{item:weakest_relaxation1}, let~$F \in \M{\mc{F}}$ be such that~$\mathbbm{1}_F \notin \P{f}$. Let~$U = V(F)$ and suppose by contradiction that~$x(U) \leq |U| - |F| - 1$ is not valid for~$\P{f}$. We show that this implies that~$\mathbbm{1}_F$ satisfy all the GSECs~$x(S) \leq |S| - \f{S}$ defining~$\P{f}$, contradicting the choice of~$F$. By the definition of RHS functions, we assume without loss of generality that~$U \cap S \neq \emptyset$.

    \noindent
    \textbf{Case~$|U \cap S| = |U|$:} We first claim that~$\f{S} \leq |S \setminus U| + |F|$. To see this, suppose by contradiction that~$\f{S} \geq |S \setminus U| + |F| + 1$. For any~$\bar{x} \in \P{f}$,~$$\bar{x}(U) \leq \bar{x}(S) \leq |S| - \f{S} \leq |U| - |F| - 1,$$ contradicting the assumption that~$x(U) \leq |U| - |F| - 1$ is not valid for~$\P{f}$. Hence,~$|F| \geq \f{S} - |S \setminus U|$, which yields~$$\mathbbm{1}_F(S) = \mathbbm{1}_F(U) = |U| - |F| \leq |U| - (\f{S} -|S \setminus U|) = |S| - \f{S}.$$
    
    \noindent
    \textbf{Case~$0 < |V(F) \cap S| < |V(F)|$:} Let~$H$ be the forest obtained by deleting from~$F$ all the vertices that are not in~$S$. By minimality of~$F$, we know that~$H \in \mc{F}$. Since~$\mathbbm{1}_H \in \{\mathbbm{1}_{F'} : F' \in \mc{F}\} \subseteq \P{f}$, it follows that~$\mathbbm{1}_F(S) = \mathbbm{1}_H(S) \leq |S| - \f{S}$.\\

    To prove item~\ref{item:weakest_relaxation2}, we can just apply item~\ref{item:weakest_relaxation1} for each minimal infeasible forest defining~$\lb{\mc{F}}$. Specifically, let~$\emptyset \subsetneq S \subseteq V$ be such that~$\lb{\mc{F}}{S} \geq 2$, and let~$F \in \M{\mc{F}}$ be such that~$V(F) = S$ and~$|F| = \lb{\mc{F}}{V(F)} - 1$. Then, item~\ref{item:weakest_relaxation1} implies that~$x(S) \leq |S| - \lb{\mc{F}}{S}$ is valid for~$\P{f}$.
\end{proof}

On the other hand, perhaps not surprisingly, the strongest possible set~$\P{f}$ that represents~$\mc{F}$ is given by the following RHS function.
\begin{definition}
    \label{def:upper_bound}
    Let~$\mc{F}$ be a nonempty family of forests. Define the RHS function
    \begin{equation*}
        u_{\mc{F}}(S) \coloneqq \min \{|S| - |E(F) \cap E(S)| : F \in \mc{F}\},
    \end{equation*}
    for each~$\emptyset \subsetneq S \subseteq V$.
\end{definition}

\begin{lemma}
    \label{lemma:strongest_relaxation}
    Let~$\mc{F}$ be a family of forests. Then:
    \begin{enumerate}[(A), leftmargin=*, align=left]
        \item $\{\mathbbm{1}_F : F \in \mc{F}\} \subseteq \P{\ub{\mc{F}}}$; \label{item:strongest_relaxation1}
        \item if~$\P{f}$ contains~$\{\mathbbm{1}_F : F \in \mc{F}\}$, then~$\P{f}$ also contains~$\P{\ub{\mc{F}}}$; and \label{item:strongest_relaxation2}
        \item if~$\mc{F}$ is edge-consistent and representable, then~$\P{\ub{\mc{F}}}$ represents~$\mc{F}$.\label{item:strongest_relaxation3}
    \end{enumerate}
\end{lemma}
\begin{proof}
    We prove items~\ref{item:strongest_relaxation1} and~\ref{item:strongest_relaxation2} jointly. Suppose that~$\{\mathbbm{1}_F : F \in \mc{F}\} \subseteq \P{f}$ and take an arbitrary set~$\emptyset \subsetneq S \subseteq V$. Since~$x(S) \leq |S| - \f{S}$ is valid for~$\{\mathbbm{1}_F : F \in \mc{F}\}$, we have that
    \begin{align}
        \f{S} & \leq \min \{|S| - \mathbbm{1}_F(S) : F \in \mc{F}\} \nonumber \\
        & = \min\{|S| - |E(F) \cap E(S)| : F \in \mc{F} \} \nonumber \\
        & = \ub{\mc{F}}{S}. \nonumber
    \end{align}
    The inequality above shows that any point~$\bar{x}$ in~$\{\mathbbm{1}_F : F \in \mc{F}\}$ satisfies~$\bar{x}(S) \leq |S| - \ub{\mc{F}}{S}$, proving~\ref{item:strongest_relaxation1}. We have also shown that, for any~$\bar{x} \in \P{\ub{\mc{F}}}$, we have~$\bar{x}(S) \leq |S| - \ub{\mc{F}}{S} \leq |S| - \f{S}$, meaning that~\ref{item:strongest_relaxation2} also holds.

    Using Theorem~\ref{thm:characterization1}, we prove~\ref{item:strongest_relaxation3} by showing that~$\P{\ub{\mc{F}}} \cap \Z^E = \P{\lb{\mc{F}}} \cap \Z^E$. Since~$\P{\lb{\mc{F}}} \cap \Z^E = \{\mathbbm{1}_F : F \in \mc{F}\}$, item~\ref{item:strongest_relaxation1} gives~$\P{\ub{\mc{F}}} \cap \Z^E \supseteq \P{\lb{\mc{F}}} \cap \Z^E$. For the other side of the inclusion, apply item~\ref{item:strongest_relaxation2} with~$\f = \lb{\mc{F}}$ to obtain that~$\P{\ub{\mc{F}}} \subseteq \P{\lb{\mc{F}}}$.
\end{proof}

Applying Lemmas~\ref{lemma:weakest_relaxation} and~\ref{lemma:strongest_relaxation} with Theorem~\ref{thm:characterization1}, we close the section with the next characterization.

\begin{theorem}
    \label{thm:characterization2}
    Let~$\mc{F}$ be an edge-consistent family of forests. Then~$\P{f}$ represents~$\mc{F}$ if and only if~$\mc{F} = \Phiset{\lb{\mc{F}}}$ and~$\P{\ub{\mc{F}}} \subseteq \P{f} \subseteq \P{\lb{\mc{F}}}$.
\end{theorem}
\begin{proof}
    By Theorem~\ref{thm:characterization1} and Lemmas~\ref{lemma:weakest_relaxation} and~\ref{lemma:strongest_relaxation}, it suffices to show that, under the assumption that~$\mc{F} = \Phiset{\lb{\mc{F}}}$, we have that~$\P{\ub{\mc{F}}} \subseteq \P{f} \subseteq \P{\lb{\mc{F}}}$ implies that~$\P{f}$ represents~$\mc{F}$. Indeed,~$\P{\ub{\mc{F}}} \cap \Z^E \subseteq \P{f} \cap \Z^E \subseteq \P{\lb{\mc{F}}} \cap \Z^E$, and from Lemmas~\ref{lemma:phi_representable} and~\ref{lemma:strongest_relaxation} we know that both~$\P{\lb{\mc{F}}}$ and~$\P{\ub{\mc{F}}}$ represents~$\mc{F}$. Hence,~$\{\mathbbm{1}_F : F \in \mc{F}\} = \P{\ub{\mc{F}}} \cap \Z^E = \P{\lb{\mc{F}}} \cap \Z^E = \P{f} \cap \Z^E$.
\end{proof}

\section{Extension and the case of vehicle routing problems}
\label{section:extension}

Building on Theorem~\ref{thm:characterization2}, we now extend our GSEC-based characterization to more general MIP formulations. Specifically, we consider formulations whose feasible regions can be represented as~$\P{f} \cap \mc{Q} \cap \Z^E$, where~$\mc{Q} \subseteq \R^E$ is associated with additional constraints that are not necessarily GSECs (note that such constraints may arise from the projection of a higher-dimensional polyhedron). 

Since the subtour elimination constraints are always valid for polytope~$\P{f}$, we assume without loss of generality that these inequalities are all valid for~$\mc{Q}$, so~$\mc{Q}$ represents a family of forests~$\mc{C}$. We further assume that the family~$\mc{C}$ is known, and our goal is to study which choices of RHS functions (if any) allow us to represent a target family of forests~$\mc{H} \subseteq \mc{C}$ as the integer vectors inside~$\P{f} \cap \mc{Q}$. The following example illustrates how this abstraction may apply in the context of Section~\ref{section:setup}.

\begin{example}
\label{example:vrp}
Consider the setup in Section~\ref{section:setup}, where~$G_0 = (V_0, E_0)$ is a complete undirected graph with~$V_0 = \{0\} \dot\cup V$ and~$E_0 = \{0v : v \in V\} \dot\cup E$. Recall that, in this case, we set~$G = (V, E) = G_0 - \{0\}$. Let~$\mc{C}_{\tsc{path}}$ be the family of all subgraphs in~$G$ whose components are paths, and observe that~$\mc{C}_{\tsc{path}}$ is represented by the polytope
\begin{equation*}
\mc{Q}_{\tsc{path}} =
\left\{x|_{E} \in \R^{E}:~
\begin{aligned}
& x(\delta_{G_0}(v)) = 2, && \forall v \in V \\
& x(S) \leq |S| - 1, && \forall \emptyset \subsetneq S \subseteq V \\
& 0 \leq x_e \leq 1 + \mb{I}(0 \in e), && \forall e \in E_0
\end{aligned}
\right\}.
\end{equation*}

\noindent
Hence, by characterizing the family of forests~$\mc{H} \subseteq \mc{C}_{\tsc{path}}$ that can be represented by~$\P{f} \cap \mc{Q}_{\tsc{path}}$, we consequently determine the types of VRPs that can be modeled as in formulation~$\vrpform{f}$ without constraint~\eqref{eq:cvrp_depot} (i.e.,~$x(\delta_{G_0}(0)) = 2k$).~\qed
\end{example}

The reason that we exclude the depot constraint~\eqref{eq:cvrp_depot} in Example~\ref{example:vrp} is that its inclusion could cause the set of forests represented by~$\Qpath$ to lose its downward-closedness property, which is essential for the following result.

\begin{proposition}
    \label{prop:characterization3}
    Let~$\mc{C}$ be a nonempty, edge-consistent and downward closed family of forests in~$G$, and let~$\mc{Q} \subseteq \R^E_+$ represent~$\mc{C}$. Let~$\mc{H} \subseteq \mc{C}$ be an edge-consistent family of forests. Then~$\P{f} \cap \mc{Q}$ represents~$\mc{H}$ if and only if
    \begin{enumerate}[(i), leftmargin=*, align=left]
        \item $\mc{H} = \Phiset{\ell_{\mc{H}, \mc{C}}} \cap \mc{C}$; and \label{item:downward_representation1}
        \item $\P{u_{\mc{H}}} \subseteq \P{f} \subseteq \P{\ell_{\mc{H}, \mc{C}}}$. \label{item:downward_representation2}
    \end{enumerate}
\end{proposition}
\begin{proof}
    Suppose that~$\P{f} \cap \mc{Q}$ represents~$\mc{H}$ and let~$\mc{F}$ be the edge-consistent family of forests represented by~$\P{f}$. Since~$\{\mathbbm{1}_F : F \in \mc{F} \cap \mc{C}\} = \{\mathbbm{1}_F : F \in \mc{H}\}$ and~$\mc{F}$,~$\mc{C}$ and~$\mc{H}$ are all edge-consistent, it follows that~$\mc{H} = \mc{F} \cap \mc{C}$.

    We first prove that~$\Phiset{\ell_{\mc{H}, \mc{C}}} \cap \mc{C} \subseteq \mc{H}$, let~$F$ be a forest in~$\mc{C} \setminus \mc{H}$. Since both~$\mc{C}$ and~$\mc{F}$ are nonempty and downward closed, we have that the empty graph belongs to~$\mc{C} \cap \mc{F} = \mc{H}$. Hence, by downward-closedness of~$\mc{C}$, there exists~$F' \in \M{\mc{H}} \cap \mc{C}$. By the definition of~$\lb{\mc{F}, \mc{C}}$,~$|F'| \leq \lb{\mc{F}, \mc{C}}{V(F')} - 1$, so~$F \notin \Phiset{\ell_{\mc{H}, \mc{C}}}$. To show the other direction of the inclusion, we use the following simple claim.\\
    
    \begin{minipage}{0.9\linewidth}
    \begin{claim}
    \label{claim:proof_minimal_subset}
    $\M{\mc{H}} \cap \mc{C} \subseteq \M{\mc{F}}$.
    \end{claim}
    \begin{proof}
    Let~$F \in \M{\mc{H}} \cap \mc{C}$. Since~$\mc{H} = \mc{F} \cap \mc{C}$, we have~$F \notin \mc{F}$. 
    Moreover, for every proper subgraph~$F' \subsetneq F$, we have~$F' \in \mc{H} \subseteq \mc{F}$. 
    Hence,~$F$ is minimal (with respect to inclusion) among the elements of~$\mc{C}$ that are not in~$\mc{F}$, 
    that is,~$F \in \M{\mc{F}}$.
    \end{proof}
    \end{minipage}\\[0.3cm]

    \noindent
    Claim~\ref{claim:proof_minimal_subset} implies that, for every~$\emptyset \subsetneq S \subseteq V$,~$\lb{\mc{H}, \mc{C}}{S} \leq \lb{\mc{F}}{S}$. Hence, by the definition of~$\Phiset$,~$\mc{H} = \mc{F} \cap \mc{C} = \Phiset{\lb{\mc{F}}} \cap \mc{C} \subseteq \Phiset{\lb{\mc{H}, \mc{C}}} \cap \mc{C}$.

    To show item~\ref{item:downward_representation2}, we observe that Claim~\ref{claim:proof_minimal_subset} also implies that, if~$F \in \M{\mc{H}} \cap \mc{C}$, then~$\mathbbm{1}_F \notin \P{f}$. By item~\ref{item:weakest_relaxation1}, inequality~$x(V(F)) \leq |V(F)| - |F| - 1$ is valid for~$\P{f}$, and therefore,~$\P{f} \subseteq \P{\lb{\mc{H}, \mc{C}}}$. Proving~$\P{\ub{\mc{H}}} \subseteq \P{f}$ is thus immediate from item~\ref{item:strongest_relaxation2} of Lemma~\ref{lemma:strongest_relaxation}.

    Let us now assume that both items~\ref{item:downward_representation1} and~\ref{item:downward_representation2} hold and, again, let~$\mc{F}$ be the edge-consistent forest represented by~$\P{f}$. Our goal is to show that~$\mc{F} \cap \mc{C} = \mc{H}$. Item~\ref{item:strongest_relaxation1} of Lemma~\ref{lemma:strongest_relaxation} gives~$\{\mathbbm{1}_F : F \in \mc{H}\} \subseteq \P{\ub{\mc{H}}} \subseteq \P{f}$, meaning that~$\{\mathbbm{1}_F : F \in \mc{H}\} \subseteq \P{f} \cap \Z^E = \{\mathbbm{1}_F : F \in \mc{F}\}$. Since~$\mc{F}$ is edge-consistent and~$\mc{H} \subseteq \mc{C}$, this implies that~$\mc{H} \subseteq \mc{F} \cap \mc{C}$. To prove the reverse inclusion, let~$F \in \mc{F} \cap \mc{C}$. Since~$\mathbbm{1}_F \in \P{f} \subseteq \P{\lb{\mc{H}, \mc{C}}}$, for every~$\emptyset \subsetneq S \subseteq V$,
    \begin{equation*}
        \mathbbm{1}_F(S) \leq |S| - \lb{\mc{H}, \mc{C}}{S}
        \iff |S| - |E(S) \cap E(F)| \geq \lb{\mc{H}, \mc{C}}{S}.
    \end{equation*}
    In particular, for every~$F' \subseteq F$ and~$S = V(F')$, we have that~$|F'| \geq \lb{\mc{H}, \mc{C}}{V(F')}$. Combining this with item~\ref{item:downward_representation1} we conclude that~$F \in \mc{H} = \Phiset{\lb{\mc{H}, \mc{C}}} \cap \mc{C}$.
\end{proof}

\paragraph{\textbf{Vehicle routing problems and componentwise feasibility.}}

In order to connect Proposition~\ref{prop:characterization3} with the VRPs discussed in Section~\ref{section:setup}, we recall that feasible solutions for these VRPs are composed of routes whose corresponding paths belong to a given family of feasible paths~$\routes$. In this sense, we introduce the following definition.
\begin{definition}
    \label{def:forest_tree}
    For any family of trees~$\mc{T}$ in~$G$, we define
    \begin{equation*}
        \label{def:f_t}
        \mc{F}(\mc{T}) \coloneqq \{F \in \Omega : T \in \mc{T},~\text{for every tree~$T \in F$} \}.
    \end{equation*}
\end{definition}

\noindent
Note that not every representable family of forests can be expressed as in Definition~\ref{def:f_t}. For example, the family of forests in~$G$ containing at most~$t \in \Z_{++}$ edges is not of the form~$\Ft{\mc{T}}$ but it can be represented with the GSEC~$x(V) \leq t$.

It follows directly from Definition~\ref{def:f_t} that we can simplify the formula for~$\lb{\mc{F}, \mc{C}}$ whenever~$\mc{F} = \Ft{\mc{T}}$.
\begin{claim}
    \label{claim:ft_lower_bound}
    Let~$\mc{C}$ be a family of forests and let~$\mc{T}$ be a family of trees in~$G$. Then, for every~$\emptyset \subsetneq S \subseteq V$,
    $$\lb{\Ft{\mc{T}}, \mc{C}}{S} = 1 + \mb{I}(S \in \B{\Ft{\mc{T}}, \mc{C}}).$$
\end{claim}
\begin{proof}
    It suffices to show that, for every~$S \in \B{\Ft{\mc{T}}, \mc{C}}$,~$\lb{\Ft{\mc{T}}, \mc{C}}{S} = 2$. Fix such a set~$S$ and let~$F \in \M{\Ft{\mc{T}}} \cap \mc{C}$ be such that~$V(F) = S$. Since~$F \notin \Ft{\mc{T}}$, there exists a tree~$T \in F$ such that~$T \notin \mc{T}$. Hence, by minimality of~$F$, we know that~$F = T$ and~$|F| = 1$, as desired.
\end{proof}

As in Ghosal et al.~\cite{ghosal2024}, let us assume that~$\routes$ contains all paths with at most one vertex. Hence, since~$\Ft{\routes}$ contains the empty graph, the minimal infeasible forests with respect to~$\Ft{\routes}$ have exactly one component (as otherwise they would not be minimal). Consider the sets~$\Cpath$ and~$\Qpath$ from Example~\ref{example:vrp}. 
Setting the target family of forests~$\mc{H}$ to~$\Ft{\routes}$ and substituting the definition of the set~$\Phiset{\lb{\mc{H}, \Cpath}}$ into Proposition~\ref{prop:characterization3}, we learn that there exists a RHS function~$\f$ such that~$\P{f} \cap \mc{Q}$ represents~$\Ft{\routes}$ if and only if
\begin{align}
    \Ft{\routes} & = \{F \in \Cpath : |F'| \geq \lb{\mc{F}(\routes), \Cpath}{V(F')},~\forall F' \subseteq F \} \nonumber \\
    & = \{F \in \Cpath : |F'| \geq 1 + \mb{I}(V(F') \in \B{\Ft{\routes}, \Cpath}),~\forall F' \subseteq F, F' \neq \emptyset \} \nonumber \\
    & = \{F \in \Cpath : 1 \geq 1 + \mb{I}(V(T) \in \B{\Ft{\routes}, \Cpath}),~\forall \text{tree~$T \subseteq F, T \neq \emptyset$} \} \nonumber \\
    & = \{F \in \Cpath : V(T) \notin \B{\Ft{\routes}, \Cpath},~\forall \text{tree~$T \subseteq F, T \neq \emptyset$}\}, \label{eq:routes_equivalence1}
\end{align}

\noindent
where the second equality follows from Claim~\ref{claim:ft_lower_bound}.

Therefore,
\begin{align}
    \routes & = \{F \in \Ft{\routes} : |F| \leq 1\} \nonumber \\
    & = \{\text{path~$P \subseteq G$} : V(P') \notin \B{\Ft{\routes}, \Cpath},~\forall \text{subpath~$P' \subseteq P, P' \neq \emptyset$}\},\label{eq:routes_equivalence2}
\end{align}
and note that~$\routes$ satisfies~\eqref{eq:routes_equivalence2} if and only if~$\Ft{\routes}$ satisfies~\eqref{eq:routes_equivalence1}.

Using essentially the same reasoning as that used to prove Lemma~\ref{lemma:phi_lower_bound},
it follows that~$\Ft{\routes}$ satisfies the above equation if and only if~$\routes$ is downward closed and it satisfies the following variant of the minimal infeasibility property:
\begin{itemize}[(I),leftmargin=*, align=left]
    \item[($\star$)] If~$P$ and~$P'$ are two paths in~$G$ with the same set of vertices, then~$P \in \M{\Ft{\routes}}$ implies~$P' \notin \routes$. \label{item:path_star}
\end{itemize}

\begin{claim}
    \label{claim:minimal_infeasible_path}
    Let~$\routes$ be a family of paths in~$G$ containing all paths with at most one vertex. Then~$\routes$ satisfies~\eqref{eq:routes_equivalence2} if and only if~$\routes$ is downward closed and satisfies property \propstar.
\end{claim}
\begin{proof}
    It is clear that if~\eqref{eq:routes_equivalence2} holds, then~$\routes$ is downward closed. To show that~$\routes$ satisfies \propstar, let~$P$ and~$P'$ be two paths in~$G$ with~$P \in \M{\Ft{\routes}}$ and~$V(P) = V(P')$. Since~$P \in \Cpath$, we know that~$V(P) \in \B{\Ft{\routes}, \Cpath}$, meaning that~$P' \notin \routes$.
    
    To prove the converse, let~$\mc{R}'$ be the set in the RHS of~\eqref{eq:routes_equivalence2}. Suppose that~$P \subseteq G$ is a path that does not belong to~$\routes$ (and thus, to~$\Ft{\routes}$). Since~$\Ft{\routes}$ contains the empty graph and~$\Cpath$ is downward closed, there exists a forest~$P' \subseteq P$ such that~$P' \in \M{\Ft{\routes}} \cap \Cpath$. Moreover, by Claim~\ref{claim:ft_lower_bound},~$P'$ is a path. Hence,~$V(P') \in \B{\Ft{\routes}, \Cpath}$ and~$P$ does not belong to~$\mc{R}'$. 
    
    To show the other side of the inclusion, assume that~$P \in \routes$. By property~\propstar, there exists no path~$P' \in \M{\Ft{\routes}}$ such that~$V(P) = V(P')$. Combining this observation with Claim~\ref{claim:ft_lower_bound} we learn that~$V(P) \notin \B{\Ft{\routes}, \Cpath}$. By downward closedness of~$\Ft{\routes}$, we can repeat the same argument for every subpath~$P''$ of~$P$, proving that~$P \in \mc{R}'$.
\end{proof}

Consequently, we obtain the following result.
\begin{corollary}
\label{corollary:ghosal}
Let~$\routes$ be a family of paths in~$G$ that contains all paths with at most one vertex, is downward closed, and satisfies \propstar. Then there exists a RHS function~$\f$ such that~$\bar{x} \in \R^{E_0}$ is feasible for~$\vrpform{\f}$ if and only if~$\bar{x}$ is the incidence vector of a solution to~$\vrpprob{\routes}$, i.e., there exists a feasible solution~$\mc{S} = \{C_1, \ldots, C_k\}$ for~$\vrpprob{\routes}$ such that, for every~$e \in E_0$,
\begin{equation}
    \label{eq:correspondence_vrp}
    \bar{x}_e = \sum_{i = 1}^k \mb{I}(e \in E(C_i)).
\end{equation}
\end{corollary}
\begin{proof}
    Suppose that~$\routes$ satisfies the conditions in the statement. By Claim~\ref{claim:minimal_infeasible_path}, $\routes$ satisfies~\eqref{eq:routes_equivalence2}. As~\eqref{eq:routes_equivalence2} is equivalent to~\eqref{eq:routes_equivalence1}, this is also equivalent to~$\Ft{\routes} = \Phiset{\lb{\Ft{\routes}, \Cpath}} \cap \Cpath$. Applying Proposition~\ref{prop:characterization3}, we learn that, for any RHS function~$\f$ such that~$\P{\ub{\Ft{\routes}}} \subseteq \P{\f} \subseteq \P{\lb{\Ft{\routes}, \Cpath}}$, the set~$\P{\f} \cap \Qpath$ represents~$\Ft{\routes}$. Therefore,
    \begin{equation*}
    \left\{x|_{E} \in \R^E :~
    \begin{aligned}
    & x(\delta_{G_0}(0)) = 2k, && \\
    & x|_E \in \P{\f} \cap \Qpath, && \\
    & 0 \leq x_e \leq 1 + \mb{I}(0 \in e), && \forall e \in E_0
    \end{aligned}
    \right\}
    \end{equation*}
    represents~$\{F \in \Ft{\routes} : |F| = k\}$, proving the statement.
\end{proof}

Since property~\propstar~is weaker than vertex-consistency (or permutation invariance, see Definition~\ref{def:vertex_consistency}), Corollary~\ref{corollary:ghosal} concretely establishes that, even when specialized for VRPs, Proposition~\ref{prop:characterization3} generalizes the result of Ghosal et al.~\cite{ghosal2024}.

\section{Subadditive functions and problem applications}
\label{section:applications}

Let~$\mc{F}$ be a family of forests. While the results in Section~\ref{subsection:rhs} establish that the strongest GSEC-based relaxation for~$\mc{F}$ is~$\P{\ub{\mc{F}}}$, computing~$\ub{\mc{F}}{S}$ may be too expensive, as it requires optimizing over~$\mc{F}$. To address this, we introduce here an approach that, under suitable conditions, allows one to easily obtain a RHS function~$\f$ such that~$\P{f} \subseteq \P{\lb{\mc{F}}}$ represents~$\mc{F}$.

Let~$g : 2^{V} \to \R_+$ be such that~$g(S) \leq |S|$, for every~$S \subseteq V$. In this section, we focus on families of forests of the form~$\mc{F}(\Theta(g))$, where
\begin{equation}
    \label{def:theta}
    \Theta(g) \coloneqq \{\text{tree~$T \subseteq G$} :  g(V(T')) \leq 1,~\forall \text{subtree~$T' \subseteq T$} \}.
\end{equation}

\noindent
It follows from Theorem~\ref{thm:characterization1} that, if~$\mc{F}(\mc{T})$ is representable and~$\mc{T}$ contains all trees with no edges, then we can assume without loss of generality that~$\mc{T} = \Thetaset{g}$.
\begin{proposition}
    \label{proposition:representable_f_g}
    Let~$g : 2^{V} \to \R_+$ be such that~$g(S) \leq |S|$, for every~$S \subseteq V$. Then~$\Ft{\Thetaset{g}}$ is representable. Moreover, if~$\mc{T}$ contains all trees with no edges and~$\Ft{\mc{T}}$ is representable, then~$\mc{T} = \Thetaset{\ell_{\Ft{\mc{T}}}}$.
\end{proposition}
\begin{proof}
    To ease notation, let~$\mc{F} = \Ft{\Thetaset{g}}$.
    To prove the first part of the statement, it suffices to show that that~$\mc{F} = \Phiset{\lb{\mc{F}}}$. By Claim~\ref{claim:ft_lower_bound}, we may write
    \begin{align*}
        \Phiset{\lb{\mc{F}}} & = \{F \in \allforests : |F'| \geq \lb{\mc{F}}{V(F')},~\forall F' \subseteq F \} \\
        & = \{F \in \allforests : 1 \geq 1 + \mb{I}(V(T) \in \B{\mc{F}}), ~\forall \text{tree~$T \subseteq F, T \neq \emptyset$}\} \\
        & = \{F \in \allforests : V(T) \notin \B{\mc{F}}, ~\forall \text{tree~$T \subseteq F, T \neq \emptyset$}\}
    \end{align*}

    \noindent
    Now, suppose that~$F$ is a forest that does not belong to~$\mc{F}$. Since~$\mc{F}$ contains the empty graph, there exists a subforest~$T \subseteq F$ such that~$T \in \M{\mc{F}}$. As~$V(T) \in \B{\mc{F}}$ and~$T$ is a tree (by minimal infeasibility), it follows that~$F \notin \Phiset{\lb{\mc{F}}}$. 
    
    Conversely, suppose that~$F \in \mc{F}$ and let~$T$ be a nonempty subtree of~$F$. Since~$\mc{F}$ is downward closed, we know that~$T \in \Thetaset{g}$. Therefore,~$g(V(T)) \leq 1$, which implies that~$V(T) \notin \B{\mc{F}}$. Indeed, suppose that there exists a tree~$T' \in \M{\mc{F}}$ such that~$V(T') = V(T)$. Then~$T' \notin \Theta(g)$ and every subtree~$T'' \subsetneq T'$ belongs to~$\Theta(g)$, meaning that~$g(V(T')) > 1$, a contradiction. This shows that~$F \in \Phiset{\lb{\mc{F}}}$.
    
    To close the proof, suppose that~$\Ft{\mc{T}}$ is representable, so~$\Ft{\mc{T}} = \Phiset{\lb{\Ft{\mc{T}}}}$. Then~$$\mc{T} = \{\text{tree~$T \in \Phiset{\lb{\Ft{\mc{T}}}}$}\} = \{\text{tree~$T \subseteq G$} : \lb{\Ft{\mc{T}}}{V(T')} \leq 1,~\forall \text{subtree~$T' \subsetneq T$}\}.$$
\end{proof}

Let~$\f$ be a RHS function and observe that~$\P{f}$ does not necessarily represent~$\Ft{\Thetaset{f}}$. For example, we might have~$\f(V) = |V|$ (so~$\P{f} = \{0\}$) while~$\Thetaset{f}$ contains a tree~$T$ in~$G$ with~$E(T) \neq \emptyset$. One can show, however, that if~$\f$ is a \emph{subadditive set function} --- that is,~$\f{A \cup B} \leq \f{A} + \f{B}$ for every~$A, B \subseteq V$ with~$A \cap B = \emptyset$\,%
\footnote{The standard definition of a subadditive set function requires this inequality to hold for all~$A, B \subseteq V$. In our setting, however, it suffices to consider only disjoint sets.}
--- then~$\P{\f}$ does represent~$\Ft{\Thetaset{\f}}$. Extending this reasoning, we obtain the following result.

\begin{proposition}
    \label{prop:subadditive_rhs_function}
    Let~$g : 2^{V} \to \R_+$ be a subadditive set function such that~$g(S) \leq |S|$, for every~$S \subseteq V$. Let~$\f$ be the RHS function given by~$\f(S) \coloneqq \max\{1, \lceil g(S) \rceil\}$, for all~$\emptyset \subsetneq S \subseteq V$. Then~$\P{\f}$ represents~$\Ft{\Thetaset{g}}$.
\end{proposition}
\begin{proof}
    Let~$F \subseteq G$ be a forest such that~$\mathbbm{1}_F \in \P{f_g} \cap \Z^E$. For every subtree~$T \in F$ and subtree~$T' \subseteq T$ with~$|V(T')| \geq 2$, we have that~$\mathbbm{1}_F(V(T')) = |V(T')| - 1 \leq |V(T')| - \f{V(T')}$, which implies that~$g(V(T')) \leq \f{V(T')} \leq 1$. This shows that~$T \in \Thetaset{g}$, and consequently,~$F \in \Ft{\Thetaset{g}}$.
    
    For the converse, suppose that~$F \in \Ft{\Thetaset{g}}$. To show that~$\mathbbm{1}_F \in \P{f}$ satisfy the GSECs, take an arbitrary set~$\emptyset \subsetneq S \subseteq V$ and let~$F'$ be the subgraph of~$F$ induced by~$S$. Since every tree~$T \in F'$ belongs to~$\Thetaset{g}$, we know that~$g(V(T)) \leq 1$. Hence, by subadditivity of~$\lceil g(\,\cdot\,) \rceil$,
    $$\mathbbm{1}_F(S) = |S| - |F'| \leq |S| - \sum_{T \in F'} \lceil g(V(T)) \rceil \leq |S| - \lceil g(S) \rceil \leq |S| - \f{S}.$$
\end{proof}

A particular subclass of subadditive set functions that will be convenient for us are the \emph{XOS} functions~\cite{feige2006maximizing}:
\begin{definition}
    \label{def:xos}
    We say that~$g : 2^{V} \to \R$ is \emph{XOS} with respect to a set~$\mc{W} \subseteq \R^{V}$ if~$g(S) = \max_{w \in \mc{W}} \{w(S)\}$, for all~$\emptyset \subsetneq S \subseteq V$. For convenience, we assume that~$g(\emptyset) = 0$.
\end{definition}
\begin{fact}
    \label{fact:xos}
    Every XOS set function~$g : 2^{V} \to \R$ is subadditive.
\end{fact}
\begin{proof}
    Let~$\mc{W}$ be such that~$g(S) = \max_{w \in \mc{W}} \{w(S)\}$, for all~$\emptyset \subsetneq S \subseteq V$.
    Let~$A, B \subseteq V$ be such that~$A \cap B = \emptyset$. Let~$\bar{w} \in \mc{W}$ be such that~$g(A \cup B) = \bar{w}(A \cup B)$. Then~$$g(A \cup B) = \bar{w}(A) + \bar{w}(B) \leq \max_{w \in \mc{W}}\{w(A)\} + \max_{w \in \mc{W}}\{w(B)\} = g(A) + g(B).$$
\end{proof}

In conclusion, to formulate a family of forests~$\Ft{\mc{T}}$ using GSECs, it suffices to find an XOS function~$g$ such that~$\mc{T} = \Thetaset{g}$ and~$g(S) \leq |S|$, for all~$S \subseteq V$. In what follows, we apply this approach to the bike sharing rebalancing problem and a robust capacitated minimum spanning tree problem. We again emphasize that these formulations cannot be obtained using the framework of Ghosal et al.~\cite{ghosal2024}.

\subsection{Bike sharing rebalancing problem}

Recall the definition of~$\routesbrp$ and~$\fbrp$ from Section~\ref{section:setup}, and let~$\Cpath$ and~$\Qpath$ be set as in Example~\ref{example:vrp}. To show that~$\vrpprob{\routesbrp}$ can be expressed as~$\vrpform{\fbrp}$, we begin with a simple lemma. Although this result was already discussed somewhat informally in~\cite{dell2014bike}, we include the proof for completeness.
\begin{lemma}
    \label{lemma:brp}
    Let~$P = (v_1, \ldots, v_\ell)$ be a path in~$G$. For each~$i \in [\ell]$, define~$D(i) \coloneqq \sum_{j \in [i]} d(v_i)$ (and~$D(0) = 0$). Moreover, define~$D_{\max}(i) \coloneqq \max_{j \in \{0, \ldots, i\}} \{D(j)\}$ and~$D_{\min}(i) \coloneqq \min_{j \in \{0, \ldots, i\}} \{D(j)\}$. The path~$P$ belongs to~$\routesbrp$ if and only if~$D_{\max}(i) - D_{\min}(i) \leq Q$, for all~$i \in [\ell]$. 
\end{lemma}
\begin{proof}
    Suppose that~$P \in \routesbrp$, meaning that there exists~$q$ such that~$0 \leq q + D(i) \leq Q$, for all~$i \in [\ell]$. Fix~$i \in [\ell]$ and note that~$q + D_{\max}(i) \leq Q$ and~$q + D_{\min}(i) \geq 0$. Therefore,~$q \geq - D_{\min}(i)$ and~$D_{\max}(i) - D_{\min}(i) \leq Q$.

    For the converse, assume that~$D_{\max}(i) - D_{\min}(i) \leq Q$, for all~$i \in [\ell]$. Set~$q = - D_{\min}(\ell)$ and observe that~$q + D(i) \geq 0$, for all~$i \in [\ell]$. Moreover,
    $$q + D(i) = D(i) - D_{\min}(\ell) \leq D_{\max}(\ell) - D_{\min}(\ell) \leq Q.$$
\end{proof}

Using Lemma~\ref{lemma:brp}, we obtain the following characterization of BRP-feasible paths.
\begin{proposition}
    \label{prop:brp}
    Let~$P = (v_1, \ldots, v_\ell)$ be a path in~$G$. Then~$P \in \routesbrp$ if and only if~$|\sum_{p = i}^j d_{v_p}| \leq Q$, for every~$0 < i \leq j \leq \ell$.
\end{proposition}
\begin{proof}
    Let~$D, D_{\max}$ and~$D_{\min}$ be as in the statement of Lemma~\ref{lemma:brp}.
    To prove the forward direction, assume that~$0 \leq q \leq Q$ is such that~$0 \leq q + D(j) \leq Q$, for all~$j \in [\ell]$. Then
    \begin{align*}
        & 0 \leq q + D(j) \leq Q \\
        \iff & 0 \leq q + D(i - 1) + (D(j) - D(i - 1)) \leq Q \\
        \iff & -q - D(i - 1) \leq D(j) - D(i - 1) \leq Q - q - D(i - 1) \\
        \implies & -Q \leq D(j) - D(i - 1) \leq Q,
    \end{align*}
    where the last line follows from~$0 \leq q + D(i - 1) \leq Q$ (recall that~$D(0) = 0$). We are thus done by the fact that~$D(j) - D(i - 1) = \sum_{p = i}^j d_{v_p}$.

    For the converse, we fix~$i \in [\ell]$ and we show that~$D_{\max}(i) - D_{\min}(i) \leq Q$ (by Lemma~\ref{lemma:brp}). Let~$j_{\max}$ and~$j_{\min}$ be such that~$D_{\max}(i) = D(j_{\max})$ and~$D_{\min}(i) = D(j_{\min})$. Suppose first that~$j_{\max} > j_{\min}$. Since~$|d(\{v_{j_{\min} + 1}, \ldots, v_{j_{\max}}\}) | \leq Q$, it follows that~$$D_{\max}(i) - D_{\min}(i) = D(j_{\max}) -  D(j_{\min}) = d((v_{j_{\min} + 1}, \ldots, v_{j_{\max}})) \leq Q.$$ On the other hand, if~$j_{\max} < j_{\min}$, we know that~$|d(\{v_{j_{\max} + 1}, \ldots, v_{j_{\min}}\}) | \leq Q$, meaning that~$$D_{\min}(i) - D_{\max}(i) = D(j_{\min}) -  D(j_{\max}) = d((v_{j_{\max} + 1}, \ldots, v_{j_{\min}})) \geq -Q.$$
\end{proof}

Now, for each~$\emptyset \subsetneq S \subseteq V$, define the XOS function~$$g_{\tsc{brp}}(S) \coloneqq \max \{d(S) / Q, -d(S) / Q \}.$$ By Proposition~\ref{prop:brp}, we have that~$\routesbrp = \Thetaset{g_{\tsc{brp}}} \cap \Cpath$. Moreover, Proposition~\ref{prop:subadditive_rhs_function} implies that~$\P{\fbrp}$ represents~$\Ft{\Thetaset{g_{\tsc{brp}}}}$, meaning that~$\P{\fbrp} \cap \Qpath$ represents~$$\Ft{\Thetaset{g_{\tsc{brp}}}} \cap \Cpath = \Ft{\Thetaset{g_{\tsc{brp}}} \cap \Cpath} = \Ft{\routesbrp}.$$ 

Therefore, every feasible solution~$\bar{x}$ to formulation~$\vrpform{\fbrp}$ corresponds to a feasible solution~$\mc{S} = \{C_1, \ldots, C_k\}$ for problem~$\vrpprob{\routesbrp}$ (where the ``correspondence'' is in the sense of Equation~\eqref{eq:correspondence_vrp}).

\subsection{Robust capacitated minimum spanning tree problem}

As in Section~\ref{section:setup}, let~$G_0 = (V_0, E_0)$ be a connected undirected graph with~$V_0 = \{0\} \dot\cup V$ and~$E_0 = \{0v : v \in V\} \dot\cup E$. Set~$G = G_0 - \{0\}$ and let $Q \in \Q_+$ be a capacity value. In the CMST~\cite{hall1992polyhedral, hall1996, uchoa2008robust}, each vertex~$v \in V$ has a demand~$d_v \in [0, Q] \cap \Q_+$, and the goal is to find a spanning tree~$T$ of~$G$, rooted at~$0$, and such that the total demand of each subtree~$T'$ hanging from~$0$ does not exceed $Q$. 

Inspired by previous work on the robust CVRP~\cite{gounaris2013robust, subramanyam2020robust, pessoa2021branch}, we now introduce the robust CMST (RCMST), where, instead of assuming that~$d \in \Q^V_+$ is deterministic, we only know that~$d$ belongs to a given \emph{uncertainty set} \hypertarget{def:uncertainty_set}{$\mc{U} \subseteq \R^V_+$}. The subtrees~$T'$ rooted at a child of~$0$ must then satisfy the robust capacity constraint~$\max_{d \in \U}\{d(V(T')\} \leq Q$.

For each~$\emptyset \subsetneq S \subseteq V$, define the XOS function~$$g_{\tsc{rcmst}}(S) \coloneqq \max_{d \in \U}\left\{d(S) / Q\right\}.$$ The set of trees in~$G$ satisfying the robust capacity constraints is given by~$\Thetaset{g_{\tsc{rcmst}}}$. Therefore, defining~$f_{\tsc{rcmst}}(\,\cdot\,) \coloneqq \max\{1, \lceil g_{\tsc{rcmst}}(\,\cdot\,) \rceil\}$, we have that~$\P{f_{\tsc{rcmst}}}$ represents~$\mc{F}_{\tsc{rcmst}} = \Ft{\Thetaset{g_{\tsc{rcmst}}}}$. In this way, we can formulate the RCMST as
\begin{subequations}
\label{form:rcmstp}
\begin{align}
    \min~~ & c^\T x \\
    \text{s.t.~~} & x(V_0) = |V|, \\
    & x(\{0\} \cup S) \leq |S|, & \forall \emptyset \subsetneq S \subseteq V, \label{ineq:rcmst_depot} \\
    & x|_E \in \P{f_{\tsc{rcmst}} ; G}, \\
    & x \in \Z^E,
\end{align}
\end{subequations}
where constraints~\eqref{ineq:rcmst_depot} enforce the subtour elimination constraints for subsets of vertices containing the depot.

When~$\U$ is a singleton, Formulation~\eqref{form:rcmstp} reduces to the CMST formulation of Hall~\cite{hall1996}. For budgeted and factor model uncertainty sets,~$g_{\tsc{rcmst}}(S)$ can be computed efficiently using the analytical solutions of Gounaris et al.~\cite{gounaris2013robust}.

\section{Concluding remarks}

In this work, we presented the first characterization of the families of forests that can be represented with GSECs. Building on this result, we generalized the framework of Ghosal et al.~\cite{ghosal2024} for VRPs and derived a GSEC-based formulation for the robust capacitated minimum spanning tree (RCMST) problem. Beyond extending previous results, our findings demonstrate the versatility of GSECs as a modeling tool for a broad class of network design and vehicle routing problems. An interesting direction for future research is to investigate the computational performance of our proposed formulation for the RCMST under different choices of uncertainty sets.

\bibliographystyle{abbrv}
\bibliography{bibliography}

\end{document}